\documentclass[12pt]{article}

\usepackage{geometry} 
\usepackage{amsthm,amssymb}
\usepackage{amsmath,amsfonts,latexsym}
\usepackage{latexsym}
\usepackage{float}
\usepackage{xcolor}
\usepackage{comment}
\usepackage{url}
\usepackage[latin1]{inputenc}
\usepackage[english]{babel}

\newtheorem{thm}{Theorem}[section]

\newtheorem{cor}[thm]{Corollary}
\theoremstyle{definition}

\newtheorem{example}[thm]{Example}

\newcommand{\F}{\mathbb{F}}

\begin{document}

\vspace*{0.5cm}

\begin{center}
{\Large\bf Doubly even self-orthogonal codes from quasi-symmetric designs}
\end{center}

\vspace*{0.5cm}

\begin{center}
					
		Dean Crnkovi\'c  (deanc@math.uniri.hr) \\[3pt]
		Doris Dumi\v ci\'c Danilovi\'c (ddumicic@math.uniri.hr) \\[3pt]
		Ana \v Sumberac (ana.sumberac@math.uniri.hr)\\[3pt]
		and \\
		Andrea \v Svob (asvob@math.uniri.hr)\\[3pt]
		{\it\small Faculty of Mathematics} \\
		{\it\small University of Rijeka} \\
		{\it\small Radmile Matej\v ci\'c 2, 51000 Rijeka, Croatia}\\
\end{center}

\vspace*{0.5cm}

\begin{abstract}
In this paper, we give a construction of doubly even self-orthogonal codes from quasi-symmetric designs. Further, we study orbit matrices of quasi-symmetric designs and give a construction of doubly even self-orthogonal codes from orbit matrices of quasi-symmetric designs of Blokhuis-Haemers type.
\end{abstract}

\bigskip



\section{Introduction} \label{intro}

Construction and classification of self-orthogonal codes is an active field of research, see e.g. \cite{classification-40, Conway, dc-om-srg, so-codes-2-des, Harada, H-T-fpf, ak-aw, Pless}. Among self-orthogonal, especially self-dual codes, doubly even codes attract special attention (see \cite{classification-40, dc-as, Miezaki, RainsSloane}). In this paper, we are dealing with a construction of doubly even self-orthogonal linear codes from incidence matrices and orbit matrices of quasi-symmetric designs. In particular, we consider constructions from quasi-symmetric designs of Blokhuis-Haemers type. 
As an illustration of the method, we constructed doubly even self-orthogonal codes from orbit matrices of quasi-symmetric designs of Blokhuis-Haemers type with parameters 2-(64,24,46). Some of the codes constructed are optimal, and also two-weight codes.

\bigskip

For computations in this paper we used Magma \cite{magma}.

\bigskip

The paper is organized as follows. In Section \ref{prelim}, we give basic definitions and theorems used in the paper. In Section \ref{code-QS}, we present constructions of doubly even self-orthogonal codes from incidence matrices of quasi-symmetric designs, including a construction from quasi-symmetric designs of Blokhuis-Haemers type, and in Section \ref{code-QS-orbmat}, we give constructions from orbit matrices of quasi-symmetric designs.

\section{Preliminaries} \label{prelim}

A $t$-$(v,k,\lambda)$ \emph{design} is a finite incidence structure \( {\mathcal{D}}=({\mathcal{P},\mathcal{B}},\mathcal{I}) \), where \( {\mathcal{P}} \) and \( {\mathcal{B}} \) are disjoint sets
and \( \mathcal{I}\subseteq {\mathcal{P}}\times {\mathcal{B}} \), with the following properties:

\begin{description}
\item [1.]\( |{\mathcal{P}}|=v \);
\item [2.]every element of \( {\mathcal{B}} \) is incident with exactly \( k \)
elements of \( {\mathcal{P}} \); 
\item [3.]every $t$ distinct elements of \( {\mathcal{P}} \) are incident
with exactly
$\lambda$ elements of \nolinebreak ${\mathcal{B}}$.
\end{description}
A $t$-$(v,k,\lambda)$ design is also a $s$-$(v,k,\lambda_s)$ design for any $0 \le s \le t$.
The elements of the set \( {\mathcal{P}} \) are called \emph{points} and the elements of the set \( {\mathcal{B}} \) are called \emph{blocks}. 
In a 2-$(v, k, \lambda)$ design every point is incident with exactly $r = \frac{\lambda (v-1)}{k-1}$ blocks, and $r$ is called the replication number of a design.
The number of blocks of a $t$-design is denoted by $b$. 
If $b=v$, a $t$-design is called \emph{symmetric}. If ${\mathcal{D}}$ is a symmetric $t$-$(v,k,\lambda)$ design, then $t \le 2$.

A \emph{complement} of a $t$-design \( {\mathcal{D}}=({\mathcal{P},\mathcal{B}},\mathcal{I}) \) is the design \( {\mathcal{D'}}=({\mathcal{P},\mathcal{B}},\mathcal{I'}) \),
where $\mathcal{I'}=({\mathcal{P}}\times {\mathcal{B}}) \backslash \mathcal{I}$. If ${\mathcal{D}}$ is a symmetric design with parameters $2$-$(v,k,\lambda)$, 
then its complement ${\mathcal{D'}}$ is a symmetric design with parameters $2$-$(v,v-k,v-2k+\lambda)$.  

Let $\mathcal{D}=(\mathcal{P},\mathcal{B},\mathcal{I})$ be a $t$-$(v,k,\lambda)$ design.
For $0 \le s < k$, $s$ is called a block intersection number of $\mathcal{D}$ if there exists $x, x' \in \mathcal{B}$ such that $|x \cap x'|=s$.
A $t$-design is called \emph{quasi-symmetric} if it has exactly two block intersection numbers $x$ and $y$ points, $x < y$. A complement of a quasi-symmetric design is also quasi-symmetric.
For reading on quasi-symmetric designs we refer the reader to \cite{QSdes-CRC, QSdes-MSS}.

The block-by-point \emph{incidence matrix} of a $t$-$(v,k,\lambda)$ design is the $b \times v$ matrix whose rows are indexed by blocks and whose columns are indexed by points, with the entry in the
row $x$ and column $P$ being 1 if $(P,x) \in \mathcal{I}$, and 0 otherwise. 

A $q$-ary \emph{linear code} $C$ of dimension $k$ for a prime power $q$, is a $k$-dimensional subspace of a vector space $\mathbb{F}_{q}^n$. 
Elements of $C$ are called codewords. If $q=2$, a code $C$ is called \emph{binary}.
Let $x=(x_1,...,x_n)$ and $y=(y_1,...,y_n)\in \mathbb{F}_q^n$. The \emph{Hamming distance} between words $x$ and $y$ is the number \ $d(x,y)=\left| \{ i  :  x_i \neq y_i \} \right| $.
The \emph{minimum distance} of the code  $C$ is defined by \ $d=\mbox{min}\{ d(x,y):  x,y\in C, \ x\neq y\}$.  The \emph{weight} of a codeword $x$ is \ $w(x)=d(x,0)=|\{i  :  x_i\neq 0 \}|$. A code for which all codewords have weight divisible by four is called \emph{doubly even}, and \emph{singly even} if all weights are even and there is at least one codeword $x$ with $w(x) \equiv  2 \mod 4$. 
For a linear code, $d=\mbox{min}\{ w(x)  :  x \in C, x \neq 0 \}.$
A  $q$-ary linear code of length $n$, dimension $k$, and minimum distance $d$ is called a $[n,k,d]_q$ code.  We may use the notation $[n,k]$ if the parameters $d$ and $q$ are unspecified. 

Let $w_{i}$ denote the number of codewords of weight $i$ in a code $C$ of 
length $n$.  The \emph{weight distribution} of $C$ is the list $[\langle i, w_{i} \rangle: 0 \leq i \leq n]$.
A \emph{one-weight code} is a code which has only one nonzero weight and a \emph{two-weight code} is a code which has only two nonzero weights. 
A linear code $C$ is called projective if the minimum distance of
its dual code $C^\perp =\{ y \ | \ y\cdot x = 0 \ {\rm for} \ {\rm all} \ x\in C \}$
is greater than 2.

The \emph{dual} code $C^\perp$ is the orthogonal complement under the standard inner product
$\langle\cdot \, ,\cdot\rangle$, i.e.\ $C^{\perp} = \{ v \in \F_{q}^n : \langle v,c\rangle=0 {\rm \ for\ all\ } c \in C \}$.  
A code $C$ is \emph{self-orthogonal} if $C \subseteq C^\perp$ and \emph{self-dual} if $C = C^\perp$. 
A doubly even self-dual code of length $n$ exists if and only if $n \equiv 0 \mod 8$, while a singly even self-dual code of length $n$ exists if and only if $n$ is even (see \cite{Harada}). 
Doubly even self-dual binary codes of lengths less or equal 40 have been completely classified (see \cite{classification-40}).
Rains in \cite{Rains} showed that the minimum weight $d$ of a self-dual code $C$ of length $n$ is bounded by $d \le 4 \lfloor \frac{n}{24} \rfloor +4$, unless $n \equiv 22 \mod 24$ when 
$d\le 4 \lfloor \frac{n}{24} \rfloor+6$. 
A self-dual code meeting the upper bound is called \emph{extremal}. We say that a code is \emph{optimal} if its minimum weight achieves the theoretical upper bound on the minimum weight of $[n,k]$ linear codes, and \emph{best known} if it has the largest minimum weight among all known codes of that length and dimension.

A \emph{generator matrix} of a linear code is a matrix whose rows form a basis for a code. It is well known that a binary $[n,k]$ code is self-orthogonal if and only if the rows of its generator matrix have even weight and are orthogonal to each other. The following theorem can be found in
\cite[Theorem 1.4.8]{Huffman}.

\begin{thm} \label{thm-self-orth}
Let $C$ be a binary linear code.
  \begin{enumerate}
    \item If C is self-orthogonal and has a generator matrix each of whose rows has weight divisible by four, then every codeword of $C$ has weight divisible by four.
    \item If every codeword of $C$ has weight divisible by four, then $C$ is self-orthogonal.
  \end{enumerate}
\end{thm}

The following statement can be found in \cite{Tonch}.

\begin{thm} \label{thm-design}
Assume that $\mathcal{D}$ is a $2$-$(v,k,\lambda)$ design with block intersection numbers $s_1,s_2,\dots,s_m.$ Denote by $C$ the binary code spanned by the block-by-point incidence matrix of 
$\mathcal{D}$. If $v \equiv 0 \mod 8,\ k \equiv 0 \mod 4$, and $s_1,s_2,\dots,s_m$ are all even, then $C$ is contained in a doubly even self-dual code of length $v$.
\end{thm}

A \emph{strongly regular graph} with parameters $(v, k, \lambda, \mu)$ is a finite simple graph on $v$ vertices, regular of degree $k$, and such that any two distinct vertices have $\lambda$ common neighbours when they are adjacent and  $\mu$  common neighbours when they are not adjacent. A projective two-weight code yields a strongly regular graph. 
Let $w_1$ and $w_2$ (suppose  $w_1 < w_2$) be two weights of a projective q-ary two-weight $[n, k, d]$ code. A strongly regular graph can be constructed as follows (see \cite{two-weight}). 
The vertices of the graph are identified by the codewords, and two vertices $u$ and $v$ are adjacent if and only if $d(u, v) = w_1$. The parameters of the strongly regular graph constructed in that way are 
$(v, K, \lambda, \mu)$, where
$$v=q^k, \qquad K=n (q-1),$$
$$\lambda = K^2+3K-q(w_1+w_2)-Kq(w_1+w_2)+q^2w_1w_2,$$
$$\mu = w_1w_2q^{2-k}=K^2+K-Kq(w_1+w_2)+q^2w_1w_2.$$

\section{Doubly even self-orthogonal codes from incidence matrices of quasi-symmetric designs} \label{code-QS}

The following theorem is a consequence of Theorem \ref{thm-design}.

\begin{thm} \label{thm_qs}
Let $\mathcal{D}$ be a quasi-symmetric design with $v \equiv 0 \mod 8,\ k \equiv 0 \mod 4$ and even block intersection numbers $x$ and $y$.
Further, let $M$ be a block-by-point incidence matrix of $\mathcal{D}$ and $C$ be a binary code spanned by the rows of $M$. 
Then $C$ is contained in a doubly even self-dual binary linear code of length $v$.
\end{thm}
\begin{proof}
The statement is an immediate consequence of Theorem \ref{thm-design}.
\end{proof}



Below we give a construction of doubly even self-orthogonal codes from quasi-symmetric designs of Blokhuis-Haemers type.

\subsection{Codes from quasi-symmetric designs of Blokhuis-Haemers type}

In 1963, Shrikhande and Raghavarao proved the following theorem.

\begin{thm} (Shrikhande and Raghavarao \cite{SR}). \label{sr}
The existence of a $2$-$(v_1, k_1, \lambda_1 )$ design ${\mathcal{D}}_1$ and a resolvable $2$-$(v_2. k_2, \lambda_2)$ design ${\mathcal{D}}_2$ with $v_2 = {v_1}k_2$ implies the existence of a 
$2$-$(v,k,\lambda)$ design ${\mathcal{D}}$ with parameters \[ v=v_1{k_2}, \ k=k_1{k_2},  \ \lambda= r_1{\lambda_2} +\lambda_1(r_2 -\lambda_2), \] where $r_i = \lambda_i(v_i -1)/(k_i -1)$, $i=1, 2$.
\end{thm}

The design ${\mathcal{D}}$ in Theorem \ref{sr} is constructed as follows. Let $P$ be a parallel class of $v_1$ blocks of $D_2$.
We label the blocks of $P$ with the points of ${\mathcal{D}}_1$, and define $b_1$ blocks of the new design ${\mathcal{D}}$ as unions of $k_1$ blocks from $P$ labeled by a block of ${\mathcal{D}}_1$.

If $q$ is a power of 2, the conditions of Theorem \ref{sr} are satisfied for any symmetric 2-$(q^2,q(q-1)/2, q(q-2)/4)$ design ${\mathcal{D}}_1$
and a resolvable 2-$(q^3,q,1)$ design ${\mathcal{D}}_2$, thus the Shrikhande-Raghavarao construction yields a design ${\mathcal{D}}$ with parameters 2-$(q^3, q^{2}(q-1)/2, q(q^3 -q^2 -2)/4)$.

In \cite{b-h}, Blokhuis and Haemers proved that if ${\mathcal{D}}_2$ is the resolvable 2-$(q^3, q, 1)$ design of the lines in $AG(3,q)$, where $q$ is a power of 2, and ${\mathcal{D}}_1$ is a symmetric
2-$(q^2,q(q-1)/2, q(q-2)/4)$ design whose blocks are maximal arcs in $AG(2,q)$, then the resulting 2-$(q^3,q^2(q-1)/2,q(q^3-q^2-2)/4)$ design ${\mathcal{D}}={\mathcal{D}}(q)$ via
the construction of Theorem \ref{sr} is quasi-symmetric with block intersection numbers $q^2(q-2)/4$ and $q^2(q-1)/4$.

In the sequel, the designs obtained by the above described construction of Blokhuis and Haemers will be called the designs of Blokhuis-Haemers type.

\begin{cor} \label{cor-b-h}
Let ${\mathcal{D}}(q)$ be a quasi-symmetric design of Blokhuis-Haemers type, where $q$ is a power of $2$, $q \ge 4$ . Then the binary code spanned by the rows of a block-by-point incidence matrix
of ${\mathcal{D}}(q)$ is doubly even and self-orthogonal. 
\end{cor}
\begin{proof}
${\mathcal{D}}(q)$ is a quasi-symmetric 2-$(q^3,q^2(q-1)/2,q(q^3-q^2-2)/4)$ design with block intersection numbers $q^2(q-2)/4$ and $q^2(q-1)/4$. The statement now follows from Theorem \ref{thm_qs}.
\end{proof}

\begin{example} \label{ex-incidence-mat}
In order to illustrate the construction described in Corollary \ref{cor-b-h}, we construct doubly even self-orthogonal binary linear codes from the quasi-symmetric designs of Blokhuis-Haemers type with parameters 2-(64,24,46) having an automorphism group of order 128, that are given in \cite{qs-64}. These designs have block intersection numbers 8 and 12. According to \cite{qs-64}, there are 2699 such designs, 2696 designs having 2-rank 13 and three designs have 2-rank equal to 12. The binary linear codes spanned by the 2696 designs having 2-rank 13 are all pairwise isomorphic, and so are the binary linear codes spanned by the three designs having 2-rank equal to 12. Thus, we obtained two doubly even self-orthogonal codes, one with parameters $[64,13,24]$ and the other with parameters $[64,12,24]$.
The code with parameters $[64,13,24]$ has minimum distance equal to the best known $[64,13]$ binary linear code, and the code with parameters $[64,12,24]$ has minimum distance one less than the best known $[64,12]$ binary linear code (see \cite{codetables}).
The codes obtained are subcodes of the dual code $C^{\perp}$ of the binary code C of length $q^3$ spanned by the incidence vectors of the lines in AG$(3,q)$ (see \cite{qs-64}). 
Since the code $C^{\perp}$ has dimension 13, the obtained code with parameters $[64,13,24]$ is equal to $C^{\perp}$, and the other one is a subspace of $C^{\perp}$ of dimension 12.
The code $C^{\perp}$ has the full automorphism group of order 23224320, isomorphic to $\Gamma L(3,4)$ (see \cite{qs-64}), and the code with parameters $[64,13,24]$ has the full automorphism group of order
368640.
\end{example}

Self-orthogonal and doubly even codes can also be constructed from point-by-block incidence matrices of designs of Blokhuis-Haemers type.

\begin{cor} \label{cor-b-h-transpose}
Let ${\mathcal{D}}(q)$ be a quasi-symmetric design of Blokhuis-Haemers type, where $q$ is a power of $2$, $q \ge 4$ . Then the binary code spanned by the rows of a point-by-block incidence matrix
of ${\mathcal{D}}(q)$ is self-orthogonal. If $q \ge 8$, the code is doubly even. 
\end{cor}
\begin{proof}
The statement follows from the fact that ${\mathcal{D}}(q)$ is a quasi-symmetric 2-$(q^3,q^2(q-1)/2,q(q^3-q^2-2)/4)$ design with the replication number $r=\frac{q(q^3-1)}{2}$.
\end{proof}

\section{Doubly even self-orthogonal codes from orbit matrices of quasi-symmetric designs} \label{code-QS-orbmat}
 
In this section, we study orbit matrices of quasi-symmetric designs and give a method of constructing self-orthogonal codes from orbit matrices of quasi-symmetric designs. In particular, we use orbit matrices of quasi-symmetric designs of Blokhuis-Haemers type to construct doubly even self-orthogonal codes.

\subsection{Orbit matrices of quasi-symmetric designs}

Orbit matrices of 2-designs have been used since the 80s of the last century for a construction of 2-designs (see e.g. \cite{dc-mop, dc-as-176, ZJ-Gaeta, ZJ-VDT-91, ZJ-TvT-78}). More recently, orbit matrices of 2-designs are used to construct self-orthogonal codes (see \cite{so-codes-2-des, H-T-fpf}). Here, using a connection between quasi-symmetric designs and strongly regular graphs, we extend these studies and give one extra condition on orbit matrices that can be applied just for quasi-symmetric designs. This additional condition can be applied to speed up a construction of quasi-symmetric designs with a prescribed automorphism group, but also for a construction of self-orthogonal codes.
Especially, in this paper we use this condition to construct doubly even self-orthogonal codes from orbit matrices of quasi-symmetric designs of Blokhuis-Haemers type.

\bigskip

Let $\mathcal{D}$ be a $2$-$(v,k,\lambda)$ design with the replication number $r$, and $G\leq {\rm Aut}(\mathcal{D})$. We denote the $G$-orbits of points by $\mathcal{P}_1,\dots,\mathcal{P}_m$,
$G$-orbits of blocks by $\mathcal{B}_1,\dots,\mathcal{B}_n$, and put $|\mathcal{P}_i|=\omega_i$, $|\mathcal{B}_j|=\Omega_j$, $1\leq i\leq m$, $1\leq j\leq n$.

Further, we denote by $\gamma_{ij}$ the number of blocks of $\mathcal{B}_j$ incident with a representative of the point orbit $\mathcal{P}_i$. The following equalities hold:
\begin{align}
&0 \leq \gamma_{ij} \leq \Omega_j, \quad  1\leq i\leq m, 1\leq j\leq n,   \label{p1}\\
&\displaystyle\sum_{j=1}^n \gamma_{ij}=r, \quad 1\leq i\leq m, \label{p2}\\
&\displaystyle\sum_{i=1}^m \frac{\omega_i}{\Omega_j} \gamma_{ij}=k, \quad 1\leq j\leq n, \label{p3}\\
&\displaystyle\sum_{j=1}^n \frac{\omega_t}{\Omega_j}\gamma_{sj}\gamma_{tj}=\lambda\omega_t+\delta_{st}\cdot(r-\lambda), \quad 1\leq s,t \leq m. \label{p4}
\end{align}

A $(m \times n)$-matrix $M = ({\gamma}_{ij})$ with entries satisfying conditions $(\ref{p1})-(\ref{p4})$ is called a point orbit matrix of a $2$-$(v,k,\lambda)$ design with orbit lengths distributions 
$(\omega_{1}, \ldots ,\omega_{m})$ and $(\Omega_{1}, \ldots ,\Omega_{n})$.

In \cite{ding-tonchev}, Ding et al. classified quasi-symmetric designs with parameters $2$-$(28,12,11)$, having an automorphism of order $7$ with the use of orbit matrices.  Recently, Kr\v cadinac and Vlahovi\' c Kruc in \cite{krcko-renata}, used orbit matrices for a construction of quasi-symmetric designs on 56 points. In the cited papers, the authors use the equations for orbit matrices of 2-designs but also develop an extra one that can be used for quasi-symmetric designs. Here, we extend these studies using a connection between quasi-symmetric designs and strongly regular graphs and give one additional condition on orbit matrices that can be applied just to quasi-symmetric designs.

In the case when $D$ is a quasi-symmetric design, one can define its block graph $\Gamma(D)$, with vertices representing the blocks such that two vertices are adjacent if the corresponding blocks intersect in $y$ points. It is well known that if $\Gamma(D)$ is a connected graph, then it is a strongly regular graph. In that case, one can use the definition and properties of orbit matrices of strongly regular graphs given in \cite{om-srg} and apply to orbit matrices of quasi-symmetric designs. Note that in \cite{B-L} Behbahani and Lam studied a construction of strongly regular graphs from orbit matrices for an automorphism group of prime order.


Let $\Gamma(D)$ be a {\sf SRG}$(b, a, c, d)$ and $A$ be its adjacency matrix. Since the vertices of the graph $\Gamma(D)$ correspond to the blocks of corresponding design $D$, we have the following. Suppose an automorphism group $G$ of $\Gamma (D)$ partitions the set of vertices $V$ into $n$ orbits $\mathcal{B}_1,\dots,\mathcal{B}_n$, with sizes $\Omega_1, \ldots , \Omega_n$, respectively. This partition of $V$ is equitable, and the quotient matrix $R = [r_{ij}]$, where $r_{ij}$ represents the number of blocks from the block orbit $\mathcal{B}_j$ that is intersecting the block from the block orbit $\mathcal{B}_i$ in $y$ points, satisfy the following conditions
\begin{align}
0 \leq r_{ij} &\leq \Omega_j - \delta_{ij}, \quad  1\leq i, j\leq n,   \label{s1}\\
\sum_{j=1}^n r_{ij}&=\sum_{i=1}^n \frac{\Omega_i}{\Omega_j} r_{ij}=a,  \label{s2}\\
\sum_{s=1}^n  \frac{\Omega_s}{\Omega_j} r_{si} r_{sj}&= \delta_{ij} (a- d ) + d \Omega_i + ( c - d) r_{ji}. \label{s3}
\end{align}

A $(n \times n)$-matrix $R = [r_{ij}]$ with entries satisfying conditions (\ref{s1}), (\ref{s2}) and (\ref{s3})
is called a row orbit matrix for a strongly regular graph with parameters $(b,a, c, d)$ and orbit lengths distribution $(\Omega_1, \ldots ,\Omega_n)$.


Since the block graph of a quasi-symmetric design is strongly regular, one can obtain a connection of a point orbit matrix of the design and a row orbit matrix of its block graph in order to obtain the equations for point orbit matrix which will be valid just for quasi-symmetric block designs.

Let $B_j\in\mathcal{B}_j$ and let us count the number of elements in the set $\mathcal{S}=\{(P,B)\in \mathcal{P}\times\mathcal{B}_{j'}\, |P\in\left\langle B\right\rangle\cap\left\langle B_j\right\rangle\}$, where $\left\langle B\right\rangle$ represents the set of points contained in the block $B$ and the same goes for $\left\langle B_j\right\rangle$. We get the following:

\begin{align}
\displaystyle\frac{1}{\Omega_j}\sum_{i=1}^m \omega_i\gamma_{ij}\gamma_{ij'} =\displaystyle\sum_{B\in\mathcal{B}_{j'}}|\left\langle B\right\rangle\cap\left\langle B_j\right\rangle\ |= r_{jj'}(y-x)+\Omega_{j'}x+(k-x)\delta_{jj'}. \label{p5}
\end{align}

With the equation (\ref{p5}) we can reduce the number of possible point orbit matrices for quasi-symmetric designs with certain parameters and prescribed orbit length distributions. In the next section, we use equation (\ref{p5}) for a construction of self-orthogonal codes and doubly even codes from orbit matrices for quasi-symmetric designs.

\subsection{Construction of self-orthogonal codes from orbit matrices of quasi-symmetric designs} 

Rows of point orbit matrices of 2-designs can be used for a construction of self-orthogonal codes (see \cite{so-codes-2-des}). In the following theorems, we give methods of constructing self-orthogonal codes from columns of point orbit matrices of quasi-symmetric designs.

\begin{thm} \label{thm-orbmat}
Let $G$ be an automorphism group of a quasi-symmetric $(v, k, \lambda)$ design $\mathcal{D}$ with block intersection numbers $x$ and $y$. Further, let $G$ act on the set of points and the set of blocks of 
$\mathcal{D}$ in orbits of the same size $\omega$. If $p$ is a prime dividing $k$, $x$ and $y$, then the columns of the point orbit matrix of the design $\mathcal{D}$ with respect to $G$ span a self-orthogonal code of length $\frac{v}{\omega}$ over the field $\mathbb{F}_{q}$, where $q=p^n$.
\end{thm}
\begin{proof}
The statement follows from  (\ref{p5}).
\end{proof}

Given an orbit matrix $M$, the rows and columns that correspond to the non-fixed points and the non-fixed blocks form a submatrix called the non-fixed part of the orbit matrix $M$.

\begin{thm} \label{thm-orbmat-non-fix}
Let $G$ be an automorphism group of a quasi-symmetric $(v, k, \lambda)$ design $\mathcal{D}$ with block intersection numbers $x$ and $y$, and $M$ be the point orbit matrix of $\mathcal{D}$ with respect to $G$. Further, let $G$ act on $\mathcal{D}$ with $f$ fixed points, $h$ fixed blocks, and all other orbits of the same size $\omega$. If a prime $p$ divides $\omega$, $y-x$ and $k-x$, then the columns of the non-fixed part of the orbit matrix $M$ span a self-orthogonal code over $\mathbb{F}_{q}$, where $q=p^n$.
\end{thm}
\begin{proof}
The group $G$ acts on the set of points of $\mathcal{D}$ in $t_v=f + \frac{v-f}{\omega}$ orbits, and on the set of blocks of $\mathcal{D}$ in $t_b=h + \frac{b-h}{\omega}$ orbits. We can assume that
$\omega_i = 1$ for $i = 1, \ldots , f$, $\omega_i = \omega$ for $i = f + 1, \ldots , t_v$, $\Omega_i = 1$ for $i = 1, \ldots , h$, $\Omega_i = \omega$ for $i = h + 1, \ldots , t_b$.
For $h + 1 \le j, s \le t_b$ we have
\begin{align*}
\sum_{i=1}^{t_v} \frac{\omega_i}{\Omega_j}\gamma_{ij}\gamma_{is}&=
\sum_{i=1}^f\frac{1}{\omega}\gamma_{ij}\gamma_{is}+\sum_{i=f+1}^{t_v} \frac{\omega}{\omega}\gamma_{ij}\gamma_{is} \stackrel{({\text by} ~~(\ref{p5}))}{=}
r_{js}(y-x)+\omega x+(k-x)\delta_{js}.
\end{align*}
For $1 \le i \le f$ and $h+1 \le j,s \le t_b$, we have $\gamma_{ij}$, $\gamma_{is} \in \{0, \omega\}$, and therefore it holds that $\gamma_{ij}\gamma_{is}\in \{0,\omega^2\}$. Hence,
\begin{align*}
\sum_{i=f+1}^{t_v} \gamma_{ij}\gamma_{is}&=
r_{js}(y-x)+\omega (x-a)+(k-x)\delta_{js},
\end{align*}
where $a=|\{i\in\{1,\ldots,h\}|\thinspace \gamma_{ij}\gamma_{is}=\omega^2 \}|$.
\end{proof}

The following theorem gives us a method of constructing doubly even self-orthogonal codes spanned by the columns of the non-fixed part of a point orbit matrix of a quasi-symmetric design of Blokhuis-Haemers type.

\begin{thm} \label{thm-orbmat-B-H}
Let $\mathcal{D}(q)$ be a quasi-symmetric design of Blokhuis-Haemers type, where $q \ge 4$. 
Further, let $G$ be an automorphism group of $\mathcal{D}(q)$, acting on $\mathcal{D}(q)$ with $f$ fixed points and $h$ fixed blocks, and all other orbits of size $2$, and $M$ be the point orbit matrix of $\mathcal{D}(q)$ with respect to $G$. 
Then the binary code spanned by the columns of the non-fixed part of the point orbit matrix $M$ is a doubly even self-orthogonal code.
\end{thm}
\begin{proof}
The design $\mathcal{D}(q)$ has parameters $2$-$(q^3,q^2(q-1)/2,q(q^3-q^2-2)/4)$ and the block intersection numbers $x=q^2(q-2)/4$ and $y=q^2(q-1)/4$. 
Since $q$ is a power of 2 and $q \ge 4$, the numbers $k=q^2(q-1)/2$, $x$ and $y$ are divisible by 4. It follows from Theorem \ref{thm-orbmat-non-fix} that the non-fixed part of the orbit matrix $M$ span a self-orthogonal code over $\mathbb{F}_{2}$.

Assume that the block orbits are ordered in a way that for their lengths it holds that $\Omega_i = 1$ for $i = 1, \ldots , h$, $\Omega_i = 2$ for $i = h + 1, \ldots , t_b$, 
where $t_b$ is the number of $G$-orbits on the blocks of $\mathcal{D}(q)$.
Similarly, let $t_v$ be the number of $G$-orbits on points and $\omega_i = 1$ for $i = 1, \ldots , f$, $\omega_i = 2$ for $i = f + 1, \ldots , t_v$.
Further, let $M$ be the point orbit matrix of $\mathcal{D}(q)$ with respect to $G$.

Let $j,s \in \{ h+1, h+2, \ldots , t_b \}$. If follows from equation (\ref{p5}) that 
$$\frac{1}{\Omega_j} \sum_{i=1}^{t_v} \omega_i \gamma_{ij}\gamma_{is}=\sum_{i=1}^f\frac{1}{2}\gamma_{ij}\gamma_{is}+\sum_{i=f+1}^{t_v} \gamma_{ij}\gamma_{is}$$ 
is divisible by 4. Let us consider the case $j=s$. Two blocks belonging to the $jth$ orbit intersect in either $x$ or $y$ points, so the cardinality of their intersection is divisible by 4. These two blocks intersect in one point of the $ith$ point orbit if $\gamma_{ij}=2$ and $1 \le i \le f$, in two points if $\gamma_{ij}=2$ and $f+1 \le i \le t_v$, and do not intersect otherwise. 

Therefore, for
$j \in \{ h+1, h+2, \ldots , t_b \}$ it holds that
$$| \{ \gamma_{ij}=2|\thinspace 1 \le i \le f \} |$$
is an even number. It follows that $\sum_{i=1}^f\frac{1}{2}\gamma_{ij}^2$ is divisible by 4, so $\sum_{i=f+1}^{t_v} \gamma_{ij}^2$ is also divisible by 4. Hence,
$$| \{ \gamma_{ij}=1|\thinspace f+1 \le i \le t_v \} |$$
is divisible by 4, which implies that the self-orthogonal binary code spanned by the columns of the non-fixed part of the orbit matrix $M$ is doubly even.  
\end{proof}

\begin{example} \label{ex-orbit-mat}
To illustrate the construction given in Theorem \ref{thm-orbmat-B-H}, we construct doubly even self-orthogonal binary linear codes from orbit matrices of involutions acting on the 2699 quasi-symmetric designs of Blokhuis-Haemers type with parameters 2-(64,24,46) having an automorphism group of order 128 (see \cite{qs-64}). These designs, that have the block intersection numbers 8 and 12, are also used in Example \ref{ex-incidence-mat}. 
The involutory automorphisms act on these 2699 designs in the following way:
\begin{itemize}
 \item fixed-point-free, 16 fixed blocks,
 \item 4 fixed points, 28 fixed blocks,
 \item 8 fixed points, 32 fixed blocks.
\end{itemize}
The information on the codes constructed is given in Table \ref{table1}. Some of the codes obtained are optimal.

\begin{table}[h]
\begin{center}
\begin{tabular}{c | c | c | c | c | c | c}
 \hline  \# fixed & code & parameters  & $|{\rm Aut}(C_i)|$ & self-      & doubly & optimal \\
          points  &      & of $C_i$    &                    & orthogonal & even   &         \\
\hline\hline
 $0$ & $C_1$ & $[32,5,16]_2$ & $21139292160$     & yes & yes & yes \\
 $0$ & $C_2$ & $[32,4,16]_2$ & $147941222252544$ & yes & yes & yes \\
 $4$ & $C_3$ & $[28,6,12]_2$ & $40320$           & yes & yes & yes \\
 $4$ & $C_4$ & $[28,5,12]_2$ & $73728$           & yes & yes & no  \\
 $8$ & $C_5$ & $[24,4,12]_2$ & $2359296$         & yes & yes & yes \\
 $8$ & $C_6$ & $[24,3,12]_2$ & $4586471424$      & yes & yes & no  \\
\hline\hline
\end{tabular}
\end{center}
\caption{\small Codes from orbit matrices of 2-(64,24,46) designs of Blokhuis-Haemers type} \label{table1}
\end{table}

The full automorphism group of the code $C_3$ with parameters $[28,6,12]$ is isomorphic to the symmetric group $S_8$. The weight distribution of $C_3$ is $[ \langle 0, 1\rangle , \langle 12, 28\rangle, \langle 16, 35\rangle ]$, and the minimum distance of $C_3^{\perp}$ is 3. Hence, $C_3$ is a projective two-weight code. The code $C_3$ was previously known, since the projective two-weight codes with parameters $[28,6,12]$ are classified in \cite{uniformly-packed}. The strongly regular graph obtained from the two-weight code $C_3$ has parameters $(64,28,12,12)$ and the full automorphism group of order 2580480, isomorphic to $Z_2^6:S_8$. Up to isomorphism, there are 15 subgroups of $Z_2^6:S_8$ which are isomorphic to $Z_2^6$. Two of these subgroups isomorphic to $Z_2^6$ act regularly on the set of vertices of the strongly regular graph with parameters $(64,28,12,12)$, one of them is normal in $Z_2^6:S_8$ and the other one is not normal. These two regular actions of the subgroups isomorphic to the elementary abelian group $Z_2^6$ correspond to difference sets (see \cite{Menon}). The adjacency matrix of a strongly regular graph with parameters $(v,k, \lambda, \lambda)$ is the incidence matrix of a symmetric design with parameters $(v,k, \lambda)$. Hence, the strongly regular graph obtained from the code $C_3$ corresponds to a symmetric $(64,28,12)$ design, which is the development of a Hadamard difference set (see \cite{difference-sets-CRC}).

The codes $C_1$, $C_2$, $C_4$, $C_5$ and $C_6$ are also two-weight codes. However, the dual codes of the codes $C_1$, $C_2$, $C_4$, $C_5$ and $C_6$ have minimum distance 2, so these codes are not projective. 
\end{example}

Under certain conditions, the rows of the non-fixed part of the point orbit matrix of a quasi-symmetric design of Blokhuis-Haemers type can also be used for a construction of doubly even self-orthogonal codes.

\begin{thm} \label{thm-orbmat-B-H-rows}
Let $\mathcal{D}(q)$ be a quasi-symmetric design of Blokhuis-Haemers type, where $q \ge 2$. 
Further, let $G$ be an automorphism group of $\mathcal{D}(q)$, acting on $\mathcal{D}(q)$ with $f$ fixed points and $h$ fixed blocks, and all other orbits of size $2$, and $M$ be the point orbit matrix of $\mathcal{D}(q)$ with respect to $G$. 
Then the binary code spanned by the rows of the non-fixed part of the point orbit matrix $M$ is self-orthogonal. If $q \ge 4$, the code is doubly even.
\end{thm}
\begin{proof}
Let $t_v$ be the number of point orbits and $t_b$ be the number of block orbits. Further, let $\Omega_i = 1$ for $i = 1, \ldots , h$, $\Omega_i = 2$ for $i = h + 1, \ldots , t_b$, 
$\omega_i = 1$ for $i = 1, \ldots , f$, and $\omega_i = 2$ for $i = f + 1, \ldots , t_v$.
Furthermore, let $M$ be the point orbit matrix of $\mathcal{D}(q)$ with respect to $G$.

Let $i,s \in \{ f+1, f+2, \ldots , t_v \}$. Then
$$\sum_{j=1}^n \frac{2}{\Omega_j}\gamma_{ij}\gamma_{sj} = \sum_{j=1}^h 2 \gamma_{ij}\gamma_{sj}+\sum_{j=h+1}^{t_b} \gamma_{ij}\gamma_{sj} = 2 \lambda+\delta_{is}\cdot(r-\lambda).$$
Since in the case of the design $\mathcal{D}(q)$ it holds that $r - \lambda = \frac{q^4 + q^3}{4}$, the binary code spanned by the rows of the non-fixed part of $M$ is self-orthogonal.

Let $q \ge 4$. Two points belonging to the $ith$ orbit, $f+1 \le i \le t_v$, belong to $\lambda$ common blocks, and $\lambda$ is even. 
Using a similar reasoning as in the proof of Theorem \ref{thm-orbmat-B-H}, we conclude that for $i \in \{ f+1, f+2, \ldots , t_v \}$ it holds that
$$| \{ \gamma_{ij}=1|\thinspace 1 \le j \le h \} |$$
is an even number. Hence, $\sum_{j=1}^h 2 \gamma_{ij}^2$ is divisible by 4. Since $2 \lambda+\delta_{is}\cdot(r-\lambda)$ is also divisible by four, 
it follows that $\sum_{j=h+1}^{t_b} \gamma_{ij}^2$ is divisible by four. Therefore, for every $i \in \{ f+1, f+2, \ldots , t_v\}$
$$| \{ \gamma_{ij}=1|\thinspace h+1 \le j \le t_b \} |$$
is divisible by 4. It follows that the binary linear code spanned by the rows of the non-fixed part of the orbit matrix $M$ is doubly even.  
\end{proof}

%
%
%
%
%
%

\vspace*{0.5cm}

\noindent {\bf Acknowledgement} \\
This work has been fully supported by Croatian Science Foundation under the project 5713.

\end{document}